\documentclass[11pt]{article}

\usepackage{hyperref}
\usepackage{amsfonts,amsmath,amssymb}
\usepackage{%hyperref,
amsthm}
\usepackage{epsfig}
\usepackage{graphicx}

\newtheorem{theorem}{Theorem}[section]
\newtheorem{proposition}[theorem]{Proposition}
\newtheorem{lemma}[theorem]{Lemma}
\newtheorem{claim}[theorem]{Claim}
\newtheorem{corollary}[theorem]{Corollary}
\newtheorem{remark}[theorem]{Remark}
\newtheorem{notation}[theorem]{Notation}

\newtheorem{definition}[theorem]{Definition}

\newtheorem{conjecture}[theorem]{Conjecture}
\newtheorem{statement}[theorem]{Statement}

\newcommand{\beq}[1]{\begin{equation}\label{#1}}
\newcommand{\enq}[0]{\end{equation}}

\newcommand{\mn}[0]{\medskip\noindent}
\newcommand{\nin}[0]{\noindent}

\newcommand{\ra}[0]{\rightarrow}

\newcommand{\A}[0]{{\cal A}}
\newcommand{\B}[0]{{\cal B}}

\newcommand{\C}[0]{{\cal C}}

\newcommand{\T}[0]{{\cal T}}
\newcommand{\W}[0]{{\cal W}}

\newcommand{\1}[0]{{\textbf{1}}}
\newcommand{\R}[0]{{\mathbb{R}}}
\newcommand{\E}[0]{{\mathbb{E}}}
\newcommand{\Cov}[0]{{\mathrm{Cov}}}

\begin{document}

\title{On the Correlation of Increasing Families}

\author{
Gil Kalai\thanks{Einstein Institute of Mathematics, Hebrew University, Jerusalem, Israel.
{\tt kalai@math.huji.ac.il}. Research supported by ERC advanced grant 320924, BSF grant 2014290, 
and NSF grant DMS-1300120.}
, Nathan Keller\thanks{Department of Mathematics, Bar Ilan University, Ramat Gan, Israel.
{\tt nathan.keller27@gmail.com}. Research
supported by the Israel Science Foundation (grant no. 402/13), the Binational US-Israel Science Foundation
(grant no. 2014290), and by the Alon Fellowship.}
, and Elchanan Mossel\thanks{Department of Statistics, University of Pennsylvania,
3730 Walnut Street,
Philadelphia, PA 19104 and Departments of Statistics and Computer Science,
U.C. Berkeley, 367 Evans Hall, Berkeley CA 94720.
 {\tt mossel@wharton.upenn.edu}. Research supported by NSF grant CCF 1320105, DOD ONR grant N00014-14-1-0823, and grant 328025 from the Simons Foundation.}
%Einstein Institute of Mathematics, Hebrew University\\
%Jerusalem 91904, Israel\\
%{\tt nkeller@math.huji.ac.il}\\
}
%\thanks{The author was partially supported by the Adams fellowship.}

%\thanks{Department of Statistics, University of Pennsylvania,
%3730 Walnut Street,
%Philadelphia, PA 19104 and Departments of Statistics and Computer Science,
%U.C. Berkeley, 367 Evans Hall, Berkeley CA 94720.
% {\tt mossel@wharton.upenn.edu}. Research supported by NSF grant CCF 1320105, DOD ONR grant N00014-14-1-0823, and grant 328025 from the Simons Foundation.}

\maketitle

\begin{abstract}

The classical correlation inequality of Harris asserts that any two monotone
increasing families on the discrete cube are nonnegatively correlated.
In 1996, Talagrand~\cite{Talagrand96} established a lower bound on the
correlation in terms of how much the two families depend simultaneously
on the same coordinates. Talagrand's method and results inspired
a number of important works in combinatorics and probability theory.

In this paper we present stronger correlation lower bounds that
hold when the increasing families satisfy natural regularity or symmetry conditions.
In addition, we present several new classes of examples for which
Talagrand's bound is tight.

A central tool in the paper is a simple lemma asserting that for monotone events
noise decreases correlation. This lemma gives also a very simple derivation of the classical FKG
inequality for product measures, and leads to a simplification of part of Talagrand's proof.

\end{abstract}

%\begin{keyword}
%Correlation, influences, Fourier-Walsh expansion
%\end{keyword}

\section{Introduction}
\label{sec:Introduction}

%{\bf EM: Do we really need both the set and function notation. Aren't functions enough?}
%
%\mn {\bf NK: I prefer to have both, since Talagrand used the set notation, and people in
%extremal combinatorics etc. prefer it. Do you think this is confusing? I'll ask also Gil
%what does he think.}
%
%{\bf EM: This is a good reason. Still the more notation one has in a paper the harder it is to read.
%I now see your reasoning and think either way is fine}

\begin{definition}
Let $\Omega_n$ denote the discrete cube $\{0,1\}^n$, and identify elements of $\Omega_n$ with
subsets of $[n]=\{1,2,\ldots,n\}$ in the natural manner. A family $\A \subset \Omega_n$ is
\emph{increasing} if $(S \in \A) \wedge (S \subset T)$ implies $T \in \A$ (alternatively, if
the characteristic function $\textbf{1}_\A$ is non-decreasing with respect to the natural
partial order on $\Omega_n$).
\end{definition}
One of the best-known correlation inequalities is Harris' inequality~\cite{Harris}
which asserts that any two increasing families $\A,\B \subset \Omega_n$ are nonnegatively correlated,
i.e., satisfy
\[
\mathrm{Cov(\A,\B)} = \mu(\A \cap \B) - \mu(\A) \mu(\B) \geq 0,
\]
where $\mu$ is the uniform measure on $\Omega_n$. In 1996, Talagrand~\cite{Talagrand96}
presented a lower bound on the correlation, in terms of \emph{influences} of the variables
on $\A,\B$.
\begin{definition}
The \emph{influence} of the $k^{th}$ variable on $\A \subset \Omega_n$ is
\[
I_k(\A) = 2\mu(\{x \in \A | x \oplus e_k \not \in \A\}),
\]
where $x \oplus e_k$ is gotten from $x$ by replacing $x_k$
by $1-x_k$.
The \emph{total influence} of $\A$ is $I(\A) = \sum_{k=1}^n I_k(\A)$.
\end{definition}
\noindent We also write $\W_1(\A,\B)=\sum_{i=1}^nI_i(\A) I_i(\B)$.

\begin{theorem}[Talagrand]\label{Thm:Talagrand-Correlation}
Let $\A,\B \subset \Omega_n$ be increasing. Then
\begin{equation}\label{Eq:Talagrand}
\Cov(\A,\B) \geq c \sum_{i=1}^n \frac{I_i(\A) I_i(\B)}
{\log (e/\sum_{i=1}^n I_i(\A) I_i(\B))}
%= c \varphi \left(\sum_{i=1}^n I_i(\A) I_i(\B) \right),
= c \varphi \left( \W_1(\A,\B) \right),
\end{equation}
where $\varphi(x) = x/\log(e/x)$ and $c$ is a universal constant.
\end{theorem}

Talagrand's theorem and the central lemma used in its proof (Lemma~\ref{Lemma:Talagrand}
below) were used in several subsequent works in combinatorics and probability theory
(e.g.,~\cite{ABGM14,FKKK15,KK13,Talagrand97}), most notably in the BKS noise sensitivity
theorem~\cite{BKS}.

So far, only two classes of tightness examples for Talagrand's lower bound are known.
Talagrand~\cite{Talagrand96} showed that his lower bound is tight when $\A,\B$ are
increasing Hamming balls, i.e., have the form $\{x:\sum x_i >t\}$, where the
thresholds $t_\A$ and $t_\B$ are chosen such that $\mu(\A)=\epsilon$ and $\mu(\B)=1-\epsilon$.
In~\cite{Keller09}, the second author presented another example, based on Ben-Or and Linial's
{\it tribes} function~\cite{BL90}, defined as follows. Partition $[n]$ into $n/r$ disjoint
sets $T_1,\ldots,T_{n/r}$ of $r$ elements, where $r \approx \log n - \log \log n$, and define $\A$ by setting $x \in \A$ iff there
exists $j$ such that $x_i=1$ for all $i \in T_j$. Let $\B$ be the \emph{dual} family
of $\A$, i.e., $x = (x_1,\ldots,x_n) \in \B$ iff $\bar{x} = (1-x_1,\ldots,1-x_n) \not \in \A$.
Then~\eqref{Eq:Talagrand} is tight for $\A,\B$.

While the two examples seem dissimilar, they share a central common feature: in both examples,
$I_i(\A) = I_i(\B)$ for all $i \in [n]$. (Moreover, in both cases $\B$ is the dual of $\A$.)
Thus, a first motivation of the current paper is seeking to find other tightness examples, especially
examples in which the relation between the structures of $\A$ and $\B$ is not so strong.

\bigskip

A second motivation is an alternative correlation lower bound, proved recently by Keller, Mossel, and
%the two latter authors jointly with Arnab
Sen~\cite{KMS14}.
\begin{theorem}\label{Thm:Alt_Bound}
Let $\A,\B \subset \Omega_n$ be increasing. Then
\begin{equation}\label{Eq:Alt_Bound}
\Cov(\A,\B) \geq c \sum_{i=1}^n \frac{I_i(\A)}{\sqrt{\log \frac{e}{I_i(\A)}}}
\frac{I_i(\B)}{\sqrt{\log \frac{e}{I_i(\B)}}} =  c \sum_{i=1}^n \psi(I_i(A)) \psi(I_i(B)),
\end{equation}
where $\psi(x) = x/\sqrt{\log(e/x)}$ and $c$ is a universal constant.
\end{theorem}
It turns out that neither of the lower bounds is strictly stronger than the other. While
Talagrand's bound is better for $\A$ being a small Hamming ball and $\B$ being its dual,
there are cases of interest for which~\eqref{Eq:Alt_Bound} is better. E.g., for $\A$ being
a small Hamming ball and $\B$ being the ``majority'' (i.e., $\{x:\sum x_i > n/2\}$),
~\eqref{Eq:Alt_Bound} is always stronger, and may be stronger even by a
multiplicative factor of $\sqrt{n/\log n}$. Hence, it is tempting to find an improved lower
bound that will combine the advantages of~\eqref{Eq:Talagrand} and~\eqref{Eq:Alt_Bound}.

\bigskip

Our main result is such a ``combined'' lower bound that holds under a weak %symmetry
regularity assumption on the families.
\begin{definition}

A family $\A \subset \Omega_n$ is
%\emph{weakly symmetric}
\emph{regular}
if all its influences are equal.
%up to a constant
%factor, i.e., there exists $c'>0$ such that $I_i(\A) \leq c' I_j(\A)$ for all $i,j$. If $c'=1$ $\A$ is {\em regular}.

\end{definition}

Note that in all examples mentioned so far (and actually, in most examples in the field,
except for ``dictatorships'' and ``juntas''), both $\A$ and $\B$ are regular. In fact, in those cases the
families are weakly symmetric, namely invariant under a transitive group of permutations on the variables.

\begin{definition}
Two increasing families $\A$ and $\B$ are {\em similar} if all ratios $I_i(\A)/I_i(\B)$ are equal,
and weakly similar if for some $c'>0$, $\max \{I_i(\A)/I_i(\B)\}
\le c'\min \{I_i(\A)/I_i(\B) \}$. Of course, all regular families are mutually similar.
\end{definition}

\begin{theorem}\label{Thm:Imp_Bound_Symmetric}

Let $\A,\B \subset \Omega_n$ be increasing and similar. Then
\begin{equation}\label{Eq:Imp_Bound1}
\Cov(\A,\B) \geq c \frac{\W_1 (\A,\B)}{\sqrt{\log \frac{e}{W_1(\A,\A)}}\sqrt{\log \frac{e}{W_1(\B,\B)}}},
\end{equation}
where $c>0$ is a universal constant.

\medskip

\noindent In particular, if $\A,\B \subset \Omega_n$ are regular, then
\begin{equation}\label{Eq:Imp_Bound1-r}
\Cov(\A,\B) \geq c \sum_{i=1}^n \frac{I_i(\A)}{\sqrt{\log \frac{e}{n I_i(\A)^2}}}
\frac{I_i(\B)}{\sqrt{\log \frac{e}{n I_i(\B)^2}}} = c \frac{I(\A) I(\B)}{n \sqrt{\log \frac{e n}{I(\A)^2}}
\sqrt{\log \frac{e n}{I(\B)^2}}}.
\end{equation}
%where $c$ is a universal constant.
\end{theorem}

When we let $c$ depend on $c'$, the
theorem extends to weakly similar increasing families $\A$ and $\B$.
It is easy to show (see Claim~\ref{Claim:Comparison-new}) that~\eqref{Eq:Imp_Bound1} is always at least as strong as both~\eqref{Eq:Talagrand}
and~\eqref{Eq:Alt_Bound}. Moreover, in some cases of interest it is strictly stronger.
For example,
in the case of $\A$ being a Hamming ball with $\mu(\A)= O(1/n)$ and $\B$ being the ``majority'',
~\eqref{Eq:Imp_Bound1} is tight, while~\eqref{Eq:Alt_Bound} is off by a factor of $\sqrt{\log n}$, and
~\eqref{Eq:Talagrand} is off by $\sqrt{\log(1/\mu(\A))}$, that may be as large as $\sqrt{n}$.

\medskip

We achieve Theorem~\ref{Thm:Imp_Bound_Symmetric} by proposing a somewhat simpler proof of
Theorem~\ref{Thm:Talagrand-Correlation} that allows to handle better
similarity and regularity  assumptions on the families.
The new proof uses a simple lemma regarding a property of the
classical \emph{noise operator}.
\begin{definition}
Let $0 \leq \rho \leq 1$. The noise operator $T_\rho: \mathbb{R}^{\Omega_n} \ra \mathbb{R}^{\Omega_n}$ is defined by
\[
T_{\rho} f(x) = \mathbb{E} [f(N_\rho x)],
\]
where $N_\rho(x)$ is obtained from $x$ by leaving each coordinate of $x$ unchanged with probability
$\rho$ and replacing it by a random value with probability $1-\rho$.
\end{definition}

\begin{lemma}\label{Lemma:Main}
Let $f,g: \Omega_n \ra \R$ be increasing. Then the function $\rho \mapsto \langle T_\rho f, g \rangle$ (where
$\langle \cdot,\cdot \rangle$ is the usual inner product on $(\Omega_n,\mu)$) is non-decreasing.
\end{lemma}
Lemma~\ref{Lemma:Main} is of independent
interest. For example, it yields an instant proof of the FKG correlation inequality~\cite{FKG} for
product measures. Indeed, as $\langle T_0 f, g \rangle = \E[f]\E[g]$ and
$\langle T_1 f, g \rangle = \mathbb{E}[fg]$, we immediately obtain
\[
\mathrm{Cov}(f,g) := \E[fg] - \E[f]\E[g] \geq 0.
\]
A consequence of Theorem~\ref{Thm:Imp_Bound_Symmetric} is:
\begin{corollary}
If $\A$ is increasing, regular and balanced and $\B$ is the majority function, then
\begin{equation}\label{Eq:cormaj}
\Cov(\A,\B) \geq c \sqrt {\log n}/\sqrt {n},
\end{equation}
where $c>0$ is a universal constant.
\end{corollary}

We note that both Theorems~\ref{Thm:Talagrand-Correlation} and~\ref{Thm:Alt_Bound} give
a weaker lower bound of $c/\sqrt n$. On the other hand, we can show that when $\A$ is the tribes family,
then $\mathrm{Cov}(\A,\mathrm{MAJ}) = \Theta(\log n /\sqrt{n})$, and we conjecture that this lower bound
holds in general. Furthermore, we conjecture that the following holds:
\begin{conjecture}
If $\A$ is increasing and balanced then there exists an increasing $\B$ represented by a linear
threshold function (i.e., $\B = \{x: \sum a_i x_i > t\}$ for nonnegative weights $a_i$),
such that $\Cov(\A,\B) \geq c \log n/\sqrt {n}$, for a universal constant $c$.
\end{conjecture}

\medskip

Our next result gives a hybrid of the
bounds~\eqref{Eq:Talagrand} and~\eqref{Eq:Alt_Bound},
under a strong symmetry condition on only one of the families $\A,\B$.

\begin{definition}
A family $\A \subset \Omega_n$ is \emph{fully symmetric} if it is invariant under the action of $\mathbb{S}_n$.
\end{definition}
For example, while the Hamming balls considered above are fully symmetric, the tribes functions are
only weakly symmetric but not fully symmetric.

\begin{theorem}\label{Thm:Imp_Bound_Asymmetric}
Let $\A \subset \Omega_n$ be increasing and $\B \subset \Omega_n$ be increasing and fully symmetric. Then
\begin{equation}\label{Eq:Imp_Bound2}
\Cov(\A,\B) \geq c_1 \sum_{i=1}^n \frac{I_i(\A)}{\sqrt{\log \frac{e}{I_i(\A)}}}
\frac{I_i(\B)}{\sqrt{\log \frac{e^2}{\W_1(\B,\B)}}} \geq c_2
\frac{\mu(\B)(1-\mu(\B))}{\sqrt{n}}  \sum_{i=1}^n \frac{I_i(\A)}{\sqrt{\log \frac{e}{I_i(\A)}}},
\end{equation}
where $c_1,c_2$ are universal constants.
\end{theorem}

It can be seen that~\eqref{Eq:Imp_Bound2} is always stronger than~\eqref{Eq:Alt_Bound}, but sometimes
weaker than~\eqref{Eq:Talagrand}. In particular, it is tight for the correlation of a small Hamming ball
and the ``majority'' family considered above. The proof of Theorem~\ref{Thm:Imp_Bound_Asymmetric} follows
Talagrand's original proof, with an enhancement that allows to handle different assumptions on
$\A,\B$ in a better way. Without additional assumptions on $\A$ and $\B$, the proof techniques of
Theorem~\ref{Thm:Imp_Bound_Asymmetric} give a new proof of Theorem~\ref{Thm:Alt_Bound}.

\medskip

Finally, we prove simple sufficient conditions for tightness of Talagrand's lower bound~\eqref{Eq:Talagrand},
and use them to show that~\eqref{Eq:Talagrand} is tight for several new examples, including (among others)
$\A$ representing an increasing
linear threshold function with low influences and $\B$ being $\A$'s dual.

\medskip

All our results extend \emph{verbatim} to bounded functions $f:\Omega_n \ra [-1,1]$, with influences
defined as
\[
I_k(f) = \E[|f(x)-f(x \oplus e_k)|].
\]
For sake of completeness, in the following sections we prove our results in the more general
form.\footnote{We note that there are several alternative generalizations of the notion of influences to
general functions on $\Omega_n$ (see, e.g.,~\cite{FHKL15} and the references therein).}

\medskip

This paper is organized as follows. In Section~\ref{sec:noise} we present the proof of
Lemma~\ref{Lemma:Main} and use it to simplify the  proof of Theorem~\ref{Thm:Talagrand-Correlation}. In Section \ref {s:similar}
we present the proof of
Theorem~\ref{Thm:Imp_Bound_Symmetric} along with some examples showing that the similarity conditions are necessary.
The (more involved) proof of Theorem~\ref{Thm:Imp_Bound_Asymmetric}
is given in Section~\ref{sec:asymmetric}. Section~\ref{sec:examples} features new tightness
examples of Talagrand's lower bound. We conclude the paper with a few open problems in
Section~\ref{sec:open}.

\section{Noise and Correlation}
\label{sec:noise}

\subsection{Preliminaries}

\begin{notation}
For $x, y \in \Omega_n$ we write $x \leq y$ if $x_i \leq y_i$ for all $i$.
A function $f: \Omega_n \ra \R$ is called \emph{increasing} or \emph{monotone} if $f(x) \leq f(y)$ whenever $x \leq y$.
\end{notation}

\nin The main technical tool used in this paper, as well as in Talagrand's work, is the Fourier-Walsh
expansion.
\begin{definition}
Let $f:\Omega_n  \rightarrow \mathbb{R}$. The Fourier-Walsh expansion of $f$ is the unique expansion
\[
f = \sum_{S \subset [n]} \alpha_S u_S,
\]
where for $T \subset [n]$,
\[
u_S(T)=(-1)^{|S \cap T|}.
\]
The coefficients $\alpha_S$ are also denoted by $\hat f(S)$, and the \emph{level} of the coefficient
$\hat f(S)$ is $|S|$.
\end{definition}
\nin
Since
$\{u_S\}_{S \subset [n]}$ is an orthonormal basis for the
function space $\mathbb{R}^{\Omega_n}$ (relative to
the usual inner product $\langle\cdot,\cdot\rangle$
with respect to uniform
measure), the representation is indeed unique, with
$\hat f(S) = \langle f ,u_S\rangle$,
and we have {\em Parseval's identity:}
\beq{Parseval}
\langle f,g\rangle=\sum \hat{f}(S)\hat{g}(S)~~~ \forall f,g.
\enq
The noise operator $T_\rho$ has a simple representation in terms of the Fourier-Walsh
expansion: For any $f = \sum_S \hat f(S) u_S$ and $\rho \in [0,1]$, we have
\begin{equation}
T_{\rho} f = \sum_S \rho^{|S|} \hat f(S) u_S.
\label{Eq:Beckner0}
\end{equation}
A standard operator we consider is the $i^{th}$ discrete derivative:
\begin{definition}
For $i \in [n]$, define $\Delta_i: \mathbb{R}^{\Omega_n} \rightarrow \mathbb{R}^{\Omega_n}$ by
$\Delta_i f(x) = \frac{1}{2} [f(x)-f(x \oplus e_i)]$.
\end{definition}
\nin It is easy to see that the Fourier expansion of $\Delta_i f$ is
\[
\Delta_i f = \sum_{S \ni i} \hat f(S) u_{S}.
\]
\nin We use the following basic properties of the Fourier expansion and the noise operator:
\begin{claim}\label{Claim:Simple}

\mn
{\rm (a)} For any $f,g: \Omega_n \ra \R$, we have $\mathrm{Cov}(f,g) = \sum_{S \neq \emptyset} \hat f(S) \hat g(S)$
(this follows immediately from~\eqref{Parseval}, since $\hat f(\emptyset) \hat g(\emptyset) = \E[f] \E[g]$).

\mn
{\rm (b)} For any increasing $f:\Omega_n \ra \R$ and any $\rho \in [0,1]$,
$T_{\rho} f$ is increasing (see, e.g.,~\cite[Proof of Proposition 4.4]{Keller10}).

\mn
{\rm (c)} For any increasing $f:\Omega_n \ra [-1,1]$ and any $i \in [n]$,
$I_i(f) = \hat f(\{i\})$ (this follows immediately from the definitions).
As a result, $\sum_i I_i(f)^2 = \sum_i \hat f(\{i\})^2 \leq 1$ by~\eqref{Parseval}.
\end{claim}
\nin For more background on the Fourier-Walsh expansion the reader is referred to~\cite{O'Donnell14}.

\subsection{Noise decreases correlation}

We present two proofs of Lemma~\ref{Lemma:Main}, which essentially shows that application of noise
reduces the correlation of increasing functions. The first proof uses the Fourier-Walsh expansion
of the noise operator, while the second uses only the chain rule and resembles the simple proof of
Russo's lemma~\cite{Russo}. Recall the formulation of the Lemma:

\mn \textbf{Lemma.}
Let $f,g: \Omega_n \ra \R$ be increasing. Then the function $h(\rho) = \langle T_\rho f, g \rangle$
is non-decreasing.

\medskip

\noindent \textbf{First Proof.} First, we note that since for any decomposition $\rho = \rho_1 \cdot \rho_2$
we have $T_\rho f = T_{\rho_2} (T_{\rho_1} f)$, and since $T_{\rho'} f$ is increasing for any
$\rho'$ (Claim~\ref{Claim:Simple}(b)), it is sufficient to show that $h'(1)$ is nonnegative.
By~\eqref{Parseval} and~\eqref{Eq:Beckner0}, we have
\[
h(\rho) = \langle T_\rho f, g \rangle = \sum_S \rho^{|S|} \hat f(S) \hat g(S),
\]
and thus,
\[
h'(1) = \sum_S |S| \hat f(S) \hat g(S) = \sum_i \sum_{S \ni i} \hat f(S) \hat g(S) =
\sum_i \langle \Delta_i f, \Delta_i g \rangle,
\]
the last equality using~\eqref{Parseval} once again. This completes the proof, as
$\Delta_i f(x) \cdot \Delta_i g(x) \geq 0$ for any $x$ by the monotonicity of $f,g$.

\medskip

\noindent \textbf{Second Proof.} Define an ``asymmetric'' noise operator $T_{\rho_1,\ldots,\rho_n}$
by
\[
T_{\rho_1,\ldots,\rho_n} f(x) = \mathbb{E} [f(N_{\rho_1,\ldots,\rho_n} x)],
\]
where $N_{\rho_1,\ldots,\rho_n}(x)$ is obtained from $x$ by leaving the $i$'th coordinate of $x$
unchanged with probability $\rho_i$ and replacing it by a random value with probability $1-\rho_i$.
As $h(\rho) = \langle T_{\rho,\rho,\ldots,\rho} f,g \rangle$, we can apply the \emph{chain rule} to assert
\[
h'(1) = \sum_{i=1}^n \frac{\partial}{\partial \rho_i} \langle T_{\rho_1,\ldots,\rho_n} f,g
\rangle \Big{|}_{(\rho_1,\ldots,\rho_n)=(1,\ldots,1)} = \sum_i \langle \Delta_i f, \Delta_i g \rangle.
\]
The rest of the argument is the same as in the first proof.

\begin{remark}
Note that as $\E[T_\rho f] = \E[f]$, an equivalent formulation of Lemma~\ref{Lemma:Main}
is that the function $\rho \mapsto \mathrm{Cov}(T_\rho f, g)$ is non-decreasing. This
formulation will be used in the proof of Theorem~\ref{Thm:Talagrand-Correlation} below.
\end{remark}

\subsection{A simpler proof of Talagrand's inequality}

\begin{notation}
From now on, for $f,g: \Omega_n \ra \R$ and $d \in \mathbb{N}$, we denote $\W_d(f,g) = \sum_{|S|=d} \hat f(S) \hat g(S)$
and $\W_d(f) = \W_d(f,f)$. Note that for increasing $\A$ and $\B$, $\W_1(1_{\A},1_{\B}) = \sum_i I_i(\A)I_i(\B)$.
\end{notation}

\nin A generalized formulation of Theorem~\ref{Thm:Talagrand-Correlation} (using Claim~\ref{Claim:Simple}(c))
is the following:

\mn \textbf{Theorem.} Let $f,g: \Omega_n \ra [-1,1]$ be increasing. Then
\begin{equation}
\mathrm{Cov}(f,g) \geq c \W_1(f,g) \left(\log (e/ \W_1(f,g)) \right)^{-1},
\end{equation}
where $c$ is a universal constant.

\mn The original proof of Theorem~\ref{Thm:Talagrand-Correlation} presented in~\cite{Talagrand96}
consists of two parts. The first part which is more complex and which we keep virtually unchanged, is proving the following
lemma, which bounds the second-level Fourier-Walsh coefficients of $f,g$ in terms of the
first-level ones:
\begin{lemma}[Talagrand]\label{Lemma:Talagrand}
Let $f,g:\Omega_n \rightarrow [-1,1]$ be increasing. Then
\[
\W_2(f,g) \leq c \W_1(f,g) \log (e/\W_1(f,g)),
\]
where $c$ is a universal constant.
\end{lemma}
This lemma appears to be of independent interest, and probably has more applications than
Theorem~\ref{Thm:Talagrand-Correlation}. A somewhat simpler (but still rather complex) proof
of the lemma is given in~\cite{KK13}, along with some generalizations.

Our argument
uses the following generalization of Lemma~\ref{Lemma:Talagrand}, proved in~\cite{BKS}
(a qualitative version) and in~\cite{KK13} (a quantitative version).
\begin{lemma}\label{Lemma:Talagrand-d}
For all increasing $f,g:\Omega_n \rightarrow [-1,1]$, and for all $2 \leq d \leq \log(e/\W_1(f,g))/2$,
\[
\W_d(f,g) \leq \frac{5e}{d} \left(\frac{2e}{d-1} \right)^{d-1} \W_1(f,g)
\left(\log (d/\W_1(f,g))\right)^{d-1}.
\]
\end{lemma}
The proof of Lemma~\ref{Lemma:Talagrand-d} is essentially the same as the proof of
Lemma~\ref{Lemma:Talagrand}. To simplify notations, we denote $C(d) = \frac{5e}{d}
\left(\frac{2e}{d-1} \right)^{d-1}$, as in the sequel we use only the fact that  $C(d) =O(1)$.

The second part of Talagrand's proof, is a somewhat complex inductive argument that deduces
the theorem from Lemma~\ref{Lemma:Talagrand}. We show now that the inductive argument in the second part of
Talagrand's proof
can be replaced with a very simple argument, using Lemma~\ref{Lemma:Main}.

\medskip

Before presenting the proof, we explain the main idea behind it, which is quite different
from the ideas used in Talagrand's proof. By Claim~\ref{Claim:Simple}(a),
the correlation can be expressed in terms of the Fourier-Walsh coefficients as
$\mathrm{Cov}(f,g) = \sum_{S \neq \emptyset} \hat f(S) \hat g(S)$. By Claim~\ref{Claim:Simple}(c),
all first-level terms in the right hand side (i.e., all terms with $|S|=1$) are nonnegative and
their total contribution is $\W_1(f,g)$. The other terms may be negative, and the assertion of
the theorem is that they cannot be ``too negative'', in the sense that their total contribution
is bounded from below by $-\W_1(f,g) + c \W_1(f,g) (\log(e/\W_1(f,g)))^{-1}$. Hence, our goal is to bound
from below the contribution of all levels $d \geq 2$.

To obtain this, we use the noise operator $T_\rho$ whose application suppresses the high-level coefficients.
By replacing $\mathrm{Cov}(f,g)$ with $\mathrm{Cov}(T_\rho f,g)$ for an appropriate choice of $\rho$,
we obtain an expression $\sum_{S \neq \emptyset} \rho^{|S|} \hat f(S) \hat g(S)$ in which the (possibly
negative) contribution of all levels $d \geq 2$ is dominated by the positive contribution of the first
level. Lemma~\ref{Lemma:Main} then allows to go back to $\mathrm{Cov}(f,g)$.

\mn \textbf{Proof of Theorem~\ref{Thm:Talagrand-Correlation}}. Let $\rho = c_0\left(\log (e/\W_1(f,g))\right)^{-1}$,
where $c_0$ is a sufficiently small constant. By Claim~\ref{Claim:Simple}(a) and
Equation~\eqref{Eq:Beckner0},
\begin{align*}
\mathrm{Cov}(T_\rho f, g) &= \sum_{S \neq \emptyset} \rho^{|S|} \hat f(S) \hat g(S) = \sum_{d \geq 1} \rho^d \W_d(f,g)= \\
&= \rho \W_1(f,g)+ \sum_{2 \leq d \leq \log(e/\W_1(f,g))/2} \rho^d \W_d(f,g) + \sum_{d > \log(e/\W_1(f,g))/2} \rho^d \W_d(f,g).
\end{align*}
By Lemma~\ref{Lemma:Talagrand-d}, for every $2 \leq d \leq \log(e/\W_1(f,g))/2$ we have
\[
\rho^d \W_d(f,g) \leq \rho^d C(d) \W_1(f,g) \left(\log (d/\W_1(f,g))\right)^{d-1} \leq
2^{-d} \rho \W_1(f,g),
\]
where the last inequality holds by the choice of $\rho$ (here is where $c_0$ should be
taken sufficiently small). For any $d > \log(e/\W_1(f,g))/2$, we use the bound
\[
\rho^d \W_d(f,g) \leq \rho^d < 2^{-d} \rho \W_1(f,g).
\]
Combining, we get
\begin{equation}\label{Eq:Tal-new}
\sum_{d \geq 2} |\rho^d \W_d(f,g)| \leq \sum_{d \geq 2} 2^{-d} \rho \W_1(f,g) \leq \rho \W_1(f,g) /2.
\end{equation}
Hence,
\[
\mathrm{Cov} (T_\rho f, g) = \rho \W_1(f,g) + \sum_{d \geq 2} \rho^d \W_d(f,g)
\geq \rho \W_1(f,g)/2 = c' \W_1(f,g) \left(\log (e/\W_1(f,g))\right)^{-1}.
\]
Therefore, by Lemma~\ref{Lemma:Main},
\[
\mathrm{Cov}(f,g) \geq \mathrm{Cov} (T_\rho f, g) \geq c' \W_1(f,g) \left(\log (e/\W_1(f,g))\right)^{-1},
\]
as asserted.

\begin{remark}
We stress that the new proof replaces only the inductive part of Talagrand's proof.
The main part of the proof (i.e., the proof of Lemma~\ref{Lemma:Talagrand}) remains
unchanged.
\end{remark}

\section{Improved Correlation Bound Under Similarity}
\label{s:similar}

In this section we present the proof of Theorem~\ref{Thm:Imp_Bound_Symmetric}, and demonstrate by several
examples that the similarity assumption in the theorem is essential.

\subsection{Proof of Theorem~\ref{Thm:Imp_Bound_Symmetric}}

\nin A generalized statement of Theorem~\ref{Thm:Imp_Bound_Symmetric} (using Claim~\ref{Claim:Simple}(c)
once again) is:

\mn \textbf{Theorem.} Let $f,g: \Omega_n \ra [-1,1]$ be increasing and similar. Then
\begin{equation}\label{Eq:Imp_Symmetric_Proof}
\mathrm{Cov(f,g)} \geq c \W_1(f,g) \left(\log (e/ \W_1(f)) \right)^{-1/2}
\left(\log (e/ \W_1(g)) \right)^{-1/2},
\end{equation}
where $c$ is a universal constant.

\mn We note that the lower bound of Theorem~\ref{Thm:Imp_Bound_Symmetric} is stronger (up to a constant) than
the bounds of Theorems~\ref{Thm:Talagrand-Correlation} and~\ref{Thm:Alt_Bound}:
\begin{claim}\label{Claim:Comparison-new}
Let $f,g: \Omega_n \ra [-1,1]$ be increasing. Then

\mn
{\rm (a)} $\W_1(f,g) \left(\log (e/ \W_1(f)) \right)^{-1/2} \left(\log (e/ \W_1(g)) \right)^{-1/2} \geq
\W_1(f,g) (\log(e/\W_1(f,g))^{-1}$,

\mn
{\rm (b)} $\W_1(f,g) \left(\log (e/ \W_1(f)) \right)^{-1/2}
\left(\log (e/ \W_1(g)) \right)^{-1/2} \geq 0.5 \sum_{i=1}^n \frac{I_i(f)}{\sqrt{\log \frac{e}{I_i(f)}}}
\frac{I_i(g)}{\sqrt{\log \frac{e}{I_i(g)}}}$.
\end{claim}

\begin{proof}
For~(a), as the numerators are equal, it is sufficient to compare the denominators. We have
\begin{align*}
\log(e/\W_1(f,g)) \geq \log(\frac{e}{\sqrt{\W_1(f)} \sqrt{\W_1(g)}}) &= \frac{1}{2}
\left(\log(\frac{e}{\W_1(f)}) + \log(\frac{e}{\W_1(g)}) \right)  \\ &\geq
\sqrt{\log (e/ \W_1(f))} \sqrt{\log (e/ \W_1(g))},
\end{align*}
where the first inequality uses Cauchy-Schwarz and the second uses the inequality between the arithmetic and
geometric means.

\mn The inequality~(b) is immediate, as for any $i$ we have
\[
\sqrt{\log \frac{e}{I_i(f)}} = \sqrt{1/2} \sqrt{\log \frac{e^2}{I_i(f)^2}} \geq \sqrt{1/2} \sqrt{\log (e/ \W_1(f))},
\]
and similarly for $g$.
\end{proof}

\mn The strategy of the proof of~\eqref{Eq:Imp_Symmetric_Proof} is similar to the simpler proof of Theorem~\ref{Thm:Talagrand-Correlation}
presented above, the only difference being the similarity assumption that allows
applying Lemma~\ref{Lemma:Talagrand-d} to $f$ and $g$ separately and then combining the
results using the Cauchy-Schwarz inequality.

\mn \textbf{Proof of Theorem~\ref{Thm:Imp_Bound_Symmetric}.} Let
\[
\rho = c'_0 \left(\log (e/ \W_1(f)) \right)^{-1/2} \left(\log (e/ \W_1(g)) \right)^{-1/2},
\]
where $c'_0$ is a sufficiently small constant. As in the proof of
Theorem~\ref{Thm:Talagrand-Correlation} above, we want to upper bound
$\sum_{d \geq 2} |\rho^d \W_d(f,g)|$. By Cauchy-Schwarz, it is sufficient to bound
\[
\sum_{d \geq 2} \rho^d \sqrt{\W_d(f)} \sqrt{\W_d(g)}.
\]
Applying the argument used above to obtain~\eqref{Eq:Tal-new} to each of the functions
$f,g$ separately (with $\rho_f = c_f \left(\log (e/ \W_1(f)) \right)^{-1}$ and
$\rho_g = c_g \left(\log (e/ \W_1(g)) \right)^{-1}$), we obtain
\begin{equation}\label{Eq:Sym1}
\rho_f^d \W_d(f) \leq 2^{-d} \rho_f \W_1(f) \qquad \mbox{and} \qquad \rho_g^d \W_d(g) \leq 2^{-d} \rho_g \W_1(g),
\end{equation}
for all $d \geq 2$. As $\sqrt{\rho_f \rho_g} = \sqrt{c_f c_g} \left(\log (e/ \W_1(f)) \right)^{-1/2}
\left(\log (e/ \W_1(g)) \right)^{-1/2} = \rho$, we can combine the inequalities in~\eqref{Eq:Sym1} and sum
over $d$ to get
\begin{equation}\label{Eq:Sym2}
\sum_{d \geq 2} \rho^d \sqrt{\W_d(f)} \sqrt{\W_d(g)} \leq \sum_{d \geq 2} 2^{-d}
\rho \sqrt{\W_1(f) \W_1(g)} \leq \rho \sqrt{\W_1(f) \W_1(g)}/2.
\end{equation}
By the similarity of $f$ and $g$, we have
\begin{equation}\label{Eq:Similarity}
\sqrt{\W_1(f) \W_1(g)} = \W_1(f,g),
\end{equation}
and thus,~\eqref{Eq:Sym2} reads
\[
\sum_{d \geq 2} \rho^d \sqrt{\W_d(f) \W_d(g)} \leq \rho \W_1(f,g)/2.
\]
Subsequently,
\begin{align}\label{Eq:Similarity2}
\begin{split}
\mathrm{Cov} (T_\rho f, g) &=  \rho \W_1(f,g) + \sum_{d \geq 2} \rho^d \W_d(f,g) \geq
\rho \W_1(f,g) - \sum_{d \geq 2} \rho^d \sqrt{\W_d(f) \W_d(g)} \geq \\ &\geq \rho
\W_1(f,g)/2 = c \W_1(f,g) \left(\log (e/\W_1(f))\right)^{-1/2} \left(\log (e/\W_1(g))\right)^{-1/2}.
\end{split}
\end{align}
The assertion now follows from Lemma~\ref{Lemma:Main}.

\begin{remark}
The proof applies almost without change if we only assume that $f,g$ are weakly similar (with respect to a
constant $c'$). The only change is that~\eqref{Eq:Similarity} holds only up to a multiplicative factor that depends
on $c'$, and that should be compensated by multiplying $\rho_f$ and $\rho_g$ by the same factor. As a
result,~\eqref{Eq:Similarity2} holds, with the constant $c$ depending on $c'$.
\end{remark}

%\begin{remark}
%We note that Theorem~\ref{Thm:Imp_Bound_Symmetric} is tight for any pair $(\A,\B)$ of
%Hamming balls. This should be compared to Theorem~\ref{Thm:Talagrand-Correlation} that is tight only
%when $\log( n I_i(\A)^2) = \Theta(\log(n I_i(\B)^2))$, and to Theorem~\ref{Thm:Alt_Bound} that is
%tight only when both $\min(\mu(\A),1-\mu(\A))$ and $\min(\mu(\B),1-\mu(\B))$ are inverse polynomially
%small in $n$.
%\end{remark}

\subsection{A few counterexamples}
\label{sec:sub:counterexamples}

As the formulation of Theorem~\ref{Thm:Imp_Bound_Symmetric} makes sense for general increasing
families, one could hope that it holds without the similarity assumption. The following examples
indicate that this is not the case. In the examples, we denote by $m_a(x_1,\ldots,x_\ell)$ an
increasing Hamming ball $\C \subset \{0,1\}^\ell$ with $\mu(\C) = a$.

\mn \textbf{Example 3.1.}\label{Ex2.1} For a small constant $a$, let $\A = m_a(x_1,\ldots,x_n)$, and let
$\B = m_{1-a}(x_1,\ldots,x_n)$ be the dual of $\A$. A direct computation (see~\cite{Talagrand96})
shows that Theorem~\ref{Thm:Talagrand-Correlation} is tight for $(\A,\B)$, as
$\mathrm{Cov}(\A,\B) = a^2$ and $\W_1(\1_\A,\1_\B) = \Theta(a^2 \log(1/a))$, where $\Theta(\cdot)$ ``hides''
a constant factor independent of $a,n$ (the latter holds since $I_i(\A) = I_i(\B) =
\Theta(a \sqrt{\log(1/a)}/\sqrt{n})$ for all $i$).

\mn Define $\A', \B' \subset \Omega_{n+1}$ by
\[
\A' = \{(x_1,\ldots,x_n,y): ((x_1,\ldots,x_n) \in \A) \vee (y=1)\}, \qquad \mbox{and} \qquad \B' = m_{1-a}(x_1,\ldots,x_n,y).
\]
We claim that the assertion of Theorem~\ref{Thm:Imp_Bound_Symmetric} does not hold for $(\A',\B')$.

\mn The influences of $\A'$ are $I_i(\A') = I_i(A)/2$ for $i \in [n]$ and $I_{n+1}(\A') = 1-a$, and
the influences of $\B'$ are $I_i(\B') \sim I_i(\B)$ (where as usual $\alpha \sim \beta$ means
$\alpha/\beta \ra 1$ as $n \ra \infty$). Hence,
\[
\W_1(\1_{\A'},\1_{\B'}) \sim \sum_{i \leq n} I_i(\A) I_i(\B) /2 + \Theta(a \sqrt{\log(1/a)}/\sqrt{n}) =
\Theta(a^2 \log(1/a)),
\]
while $\W_1(\1_{\A'}) = \Theta(1)$ and $\W_1(\1_{\B'}) \sim \W_1(\1_{\B}) = \Theta(a^2 \log(1/a))$.
Hence,
\begin{equation}\label{Eq:Counterexample1}
\frac{\W_1(\1_{\A'},\1_{\B'})}{\sqrt{\log(1/\W_1(\1_{\A'}))}\sqrt{\log(1/\W_1(\1_{\B'}))}} =
\Theta(a^2 \sqrt{\log(1/a)}).
\end{equation}
On the other hand, we claim that $\mathrm{Cov}(\A',\B') \sim a^2$. To see this, let $z$ be a new
variable independent of all others, denote $\B'' = m_{1-a}(x_1,\ldots,x_n,z)$, and consider
$\A',\B''$ as subsets of $\{0,1\}^{n+2}$. As $\A'$ does not depend on $z$ and $\B''$ does not
depend on $y$, a direct computation yields $\mathrm{Cov}(\A',\B'') \sim \mathrm{Cov}(\A,\B)/2 = a^2/2$.
Since $\E(\B'')=\E(\B')$, we have
\[
|\mathrm{Cov}[\A',\B'] - \mathrm{Cov}[\A',\B'']| = |\E[\1_{\A'}(\1_{\B'} - \1_{\B''})]| \leq
\Pr[\1_{\B'} \neq \1_{\B''}] =O(n^{-1/2}),
\]
and thus,
\begin{equation}\label{Eq:Counterexample2}
\mathrm{Cov}[\A',\B'] \sim a^2/2.
\end{equation}
Comparing Equations~\eqref{Eq:Counterexample1} and~\eqref{Eq:Counterexample2}, we see
that~\eqref{Eq:Imp_Bound1} fails for $(\A',\B')$, as asserted.

\begin{remark}
Note that the family $\B'$ in the example is not only regular %(in fact, weakly symmetric)
but even fully symmetric.
This shows that a symmetry assumption on only one of the families is insufficient. A weaker
bound that does hold when one of the families is fully symmetric is
Theorem~\ref{Thm:Imp_Bound_Asymmetric}.
\end{remark}

%A few years ago, the first author (unpublished) suggested that if
%Theorem~\ref{Thm:Imp_Bound_Symmetric} is too strong, then the following weaker
%bound is plausible.

Since the conclusion of
Theorem~\ref{Thm:Imp_Bound_Symmetric} is not true in general, we can ask about the following weaker
bound.

\begin{statement}\label{St:Wrong_Bound}
Let $\A,\B \subset \Omega_n$ be increasing. Then
\begin{equation}\label{Eq:Wrong_Bound1}
\Cov(\A,\B) \geq c \frac{\W_1(\1_\A,\1_\B)}{\log (e/ \W_1(\1_\A) \W_1(\1_\B))},
\end{equation}
where $c$ is a universal constant.
\end{statement}

\nin Recall that Theorem~\ref{Thm:Imp_Bound_Symmetric} strengthens Theorem~\ref{Thm:Talagrand-Correlation}
by decreasing the denominator of the right hand side twice: First, it replaces $\W_1(\1_\A,\1_\B)$
inside the logarithm by $\sqrt{\W_1(\1_\A)} \sqrt{\W_1(\1_\B)}$, applying Cauchy-Schwarz. Second,
it replaces the arithmetic mean $\log \left(e/ \sqrt{\W_1(\1_\A)} \sqrt{\W_1(\1_\B)} \right)$ by the geometric
mean $\sqrt{\log (e/ \W_1(\1_\A))} \sqrt{\log (e/ \W_1(\1_\B))}$. Statement~\ref{St:Wrong_Bound}
suggests to make only the first step.

\mn While the families $(\A,\B)$ of Example~3.1 satisfy Statement~\ref{St:Wrong_Bound}, the
following example shows that Statement~\ref{St:Wrong_Bound} is false, even under an additional
assumption that one of the families is regular.

\mn \textbf{Example 3.2.} For a small constant $a$, let $\A,\B$ be $\A = m_a(x_1,\ldots,x_n)$,
$\B = m_{1-a}(x_1,\ldots,x_n)$ as
in Example~3.1, and let $\C = m_{1/2}(y_1,\ldots,y_\ell)$, where $\ell=\ell(n)$ is chosen such that
$I_i(\C) = I_j(\B)$ for all $i,j$. Define $\A',\B' \subset \Omega_{n+\ell}$ by
$\A' = \{(x_1,\ldots,x_n,y_1,\ldots,y_\ell): ((x_1,\ldots,x_n) \in \A) \vee (y_1=1)\}$ and
$\B' = \{(x_1,\ldots,x_n,y_1,\ldots,y_\ell): ((x_1,\ldots,x_n) \in \B) \wedge (y_1,\ldots,y_\ell) \in \C\}$.
Note that by the choice of $\ell$, $\B'$ is regular.

We claim that~\eqref{Eq:Wrong_Bound1} fails for $(\A',\B')$. Indeed, a computation similar to that of
Example~3.1 shows that $\mathrm{Cov}(\A',\B') \sim a^2/4$, while $\W_1(\1_{\A},\1_{\B}) =
\Theta(a^2 \log(1/a))$ and $\W_1(\1_{\A}) = \W_1(\1_{\B}) = \Theta(1)$.
Hence, the right hand side of~\eqref{Eq:Wrong_Bound1} is $\Theta (a^2 \log(1/a))$ which is significantly
larger than the left hand side $(\Theta(a^2))$, rendering~\eqref{Eq:Wrong_Bound1} false.

\begin{remark}
We note that the same example, with $a = n^{-\alpha}$ for $\alpha \in (0,1/2)$, shows that in
Theorem~\ref{Thm:Imp_Bound_Asymmetric}, the assumption that $\B$ is \emph{fully symmetric} cannot be
replaced by assuming that $\B$ is merely regular. Indeed, for $(\A',\B')$ of the example, the
right hand side of~\eqref{Eq:Imp_Bound2} is
\[
c \sum_{i=1}^{n+\ell} \frac{I_i(\A)}{\sqrt{\log \frac{e}{I_i(\A)}}}
\frac{I_i(\B)}{\sqrt{\log \frac{e^2}{n I_i(\B)^2}}} = \Theta \left(\frac{a^2 \log(1/a)}{\sqrt{\log n}} \right) =
\Theta(a^2 \cdot \alpha \sqrt{\log n}),
\]
which is asymptotically larger than $\mathrm{Cov}(\1_{\A'},\1_{\B'}) = \Theta(a^2)$.
\end{remark}

%\mn An interesting problem is to find sufficient conditions on $\A,\B$ under which~\eqref{Eq:Imp_Bound1}
%holds. It is clear that one such condition is existence of $c$ such that $I_i(\A) = c I_i(\B)$ for
%all $i$. This condition (which holds, e.g., if both $\A$ and $\B$ is weakly symmetric) allows to apply
%the proof of Theorem~\ref{Thm:Imp_Bound_Symmetric} {\it verbatim}. It would be nice to find less restrictive
%conditions, or alternative lower bounds that improve on both~\eqref{Eq:Talagrand} and~\eqref{Eq:Alt_Bound}
%and hold in greater generality.

\section{An Asymmetric Correlation Bound}
\label{sec:asymmetric}

In this section we present the proof of Theorem~\ref{Thm:Imp_Bound_Asymmetric}. This proof follows the
original proof strategy of Talagrand~\cite{Talagrand96}, with a few enhancements that allow to handle
in a better way different assumptions on $\A$ and $\B$. Due to this feature, we refer to the result as
an ``asymmetric'' correlation bound.

\subsection{A few Lemmas}

In~\cite{Talagrand96}, the following simple lemma is proved and deployed.
\begin{lemma}\cite[Lemma 4.1]{Talagrand96}\label{Lemma:T-Simple}
The function $\varphi(x) = \frac{x}{\log(e/x)}$ is increasing and convex in $(0,1)$, and for all
$0 < u \leq v < 1$ we have
\begin{equation}\label{Eq:T-Simple1}
\varphi(v) \leq \varphi(u) + \frac{2(v-u)}{\log(e/v)}.
\end{equation}
\end{lemma}
\nin We shall use Lemma~\ref{Lemma:T-Simple}, along with the following two strengthenings.
\begin{lemma}\label{Lemma:T-Simple2}
For any $n \in \mathbb{N}$, the function $\psi_n(x)=\frac{x}{\sqrt{\log(e^3/nx^2)}}$ is increasing and convex in
$(0,1/\sqrt{n})$,  and for any $0 < u \leq v < 1/\sqrt{n}$ we have
\[
\psi_n(v) \leq \psi_n(u) + \frac{2(v-u)}{\sqrt{\log \frac{e^3}{n((v+u)/2)^2}}}.
\]
\end{lemma}

\begin{lemma}\label{Lemma:T-Simple3}
The function $\psi(x)=x/\sqrt{\log(e^2/x)}$ is increasing and convex in $(0,1)$, and for any
$0 < u \leq v < 1$ we have
\[
\psi(v) \leq \psi(u) + \frac{1.5(v-u)}{\sqrt{\log
\frac{e^2}{(v+u)/2}}}.
\]
\end{lemma}

\nin We note that the main advantage of Lemmas~\ref{Lemma:T-Simple2} and~\ref{Lemma:T-Simple3}
over Lemma~\ref{Lemma:T-Simple} is replacement of $v$ by $(v+u)/2$ in the denominator.
This makes the proof of these lemmas a bit more complex than Talagrand's proof of
Lemma~\ref{Lemma:T-Simple}. For sake of completeness, we present the proof of both lemmas below.
\begin{proof}[Proof of Lemma~\ref{Lemma:T-Simple2}.]
We have
\begin{align*}
\psi'_n(x)&=(\log(e^3/nx^2))^{-1/2} + x \cdot (-1/2) \cdot
(\log(e^3/nx^2))^{-3/2} \cdot (nx^2/e^3) \cdot (-2e^3/nx^3) \\
&= (\log(e^3/nx^2))^{-1/2} (1+ (\log(e^3/nx^2))^{-1}) =h(x)(1+h(x)^2),
\end{align*}
where $h(x)=\log(e^3/nx^2)^{-1/2}$. As $h(x)$ is nonnegative and increasing in $(0,1/\sqrt{n})$,
it follows that $\psi'_n$ is nonnegative and increasing, and thus $\psi_n$ is increasing
and convex. Furthermore, we have $h(x) \leq 1$ and thus, $\psi'_n(x) \leq 2h(x)$. Hence,
\begin{equation}\label{Eq7.1}
\psi_n(v) = \psi_n(u) + \int_u^v \psi'_n(x) dx \leq \psi_n(u) + \int_u^v 2h(x) dx.
\end{equation}
Now, we claim that $h(x)$ is concave in $(0,1/\sqrt{n})$. Indeed, we have
$h'(x) = x^{-1} \log(e^3/nx^2)^{-3/2}$
%\[
%h'(x)= (-1/2) \cdot (\log(e^3/nx^2)^{-3/2}) \cdot (nx^2/e^3) \cdot
%(-2e^3/nx^3) = x^{-1} \log(e^3/nx^2)^{-3/2},
%\]
and
\begin{align*}
h''(x)&= -x^{-2} \log(e^3/nx^2)^{-3/2} + x^{-1}
\cdot 3x^{-1} \log(e^3/nx^2)^{-5/2} \\
&= x^{-2} \log(e^3/nx^2)^{-3/2} \left(-1 + 3 \log(e^3/nx^2)^{-1}
\right) < 0,
\end{align*}
where the last inequality holds since $\log(e^3/nx^2)^{-1}<1/3$ for
all $x \in (0,1/\sqrt{n})$. Thus, by concavity of $h$,~\eqref{Eq7.1}
implies
\[
\psi_n(v) \leq \psi_n(u) + \int_u^v 2 h(x) dx \leq \psi_n(u) + 2(v-u)
h((v+u)/2) = \psi_n(u) +
\frac{2(v-u)}{\sqrt{\log \frac{e^3}{n((v+u)/2)^2}}},
\]
as asserted.
\end{proof}

\medskip

\begin{proof}[Proof of Lemma~\ref{Lemma:T-Simple3}.]
We have
\begin{align*}
\psi'(x)&=(\log(e^2/x))^{-1/2} + x \cdot (-1/2) \cdot
(\log(e^2/x))^{-3/2} \cdot (x/e^2) \cdot (-e^2/x^2) \\
&= (\log(e^2/x))^{-1/2} (1+ \frac12(\log(e^2/x))^{-1}) =h(x)(1+h(x)^2/2),
\end{align*}
where $h(x)=\log(e^2/x)^{-1/2}$. As $h(x)$ is nonnegative and increasing in $(0,1)$,
it follows that $\psi'$ is nonnegative and increasing, and thus $\psi$ is increasing
and convex. Furthermore, we have $h(x) \leq 1$ and thus, $\psi'(x) \leq 1.5h(x)$. Hence,
\begin{equation}\label{Eq7.1-2}
\psi(v) = \psi(u) + \int_u^v \psi'(x) dx \leq \psi(u) + \int_u^v 1.5h(x) dx.
\end{equation}
Now, we claim that $h(x)$ is concave in $(0,1)$. Indeed, we have
$h'(x) = 0.5 x^{-1} \log(e^2/x)^{-3/2}$
%\[
%h'(x)= (-1/2) \cdot (\log(e^3/nx^2)^{-3/2}) \cdot (nx^2/e^3) \cdot
%(-2e^3/nx^3) = x^{-1} \log(e^3/nx^2)^{-3/2},
%\]
and
\begin{align*}
h''(x)&= 0.5 \left(-x^{-2} \log(e^2/x)^{-3/2} + x^{-1}
\cdot 1.5x^{-1} \log(e^2/x)^{-5/2} \right) \\
&= 0.5x^{-2} \log(e^2/x)^{-3/2} \left(-1 + 1.5 \log(e^2/x)^{-1}
\right) < 0,
\end{align*}
where the last inequality holds since $\log(e^2/x)^{-1}<1/2$ for
all $x \in (0,1)$. Thus, by concavity of $h$,~\eqref{Eq7.1-2}
implies
\[
\psi(v) \leq \psi(u) + \int_u^v 1.5 h(x) dx \leq \psi(u) + 1.5(v-u)
h((v+u)/2) = \psi(u) +
\frac{1.5(v-u)}{\sqrt{\log \frac{e^2}{((v+u)/2)^2}}},
\]
as asserted.
\end{proof}

\mn Another simple but important lemma from~\cite{Chang,Talagrand96} (see
also~\cite[Remark 5.28]{O'Donnell14}) we use is the
following:
\begin{lemma}\cite[Proposition 2.2]{Talagrand96}\label{Lemma:T-Classic}
For any $f: \Omega_n \ra [-1,1]$, with $\E[|f|] \leq 1/2$, we have
\[
\sum_{i=1}^n \hat f(\{i\})^2 \leq c \E[|f|]^2 \log(e/\E[|f|]),
\]
where $c$ is an absolute constant.
\end{lemma}
\nin As noted in~\cite{Talagrand96}, the following is an immediate corollary.
\begin{corollary}\label{Cor:T-Classic}
For any $f: \Omega_n \ra [-1,1]$, and for any $k$, we have
\[
\sum_{i \neq k} \hat f(\{i,k\})^2 \leq c I_k(f)^2 \log(e/I_k(f)).
\]
\end{corollary}

\subsection{Proof of Theorem~\ref{Thm:Imp_Bound_Asymmetric}}

A generalized statement of the theorem is the following.

\mn \textbf{Theorem.}
Let $f:\Omega_n \ra [-1,1]$ be increasing and fully symmetric and $g: \Omega_n \ra [-1,1]$ be increasing. Then
\begin{equation}
\mathrm{Cov}(f,g) \geq c \sum_{i=1}^n \frac{I_i(f)}{\sqrt{\log \frac{e^3}{nI_i(f)^2}}}
\frac{I_i(g)}{\sqrt{\log \frac{e^2}{I_i(g)}}} = \sum_{i=1}^n \psi_n(I_i(f)) \psi(I_i(g)),
\end{equation}
where $\psi_n(x)=\frac{x}{\sqrt{\log(e^3/nx^2)}}$, $\psi(x) = \frac{x}{\sqrt{\log(e^2/x)}}$, and
$c$ is a universal constant.

\mn \textbf{Proof.}
The proof is by induction on $n$. The case $n=1$ is trivial, since in this case:
\[
\Cov(f,g) = I_1(f) I_1(g) \geq \psi_1(I_1(f)) \psi(I_1(g)).
\]
We now prove the induction step. We choose
to induct on the coordinate $j$ such that $I_j(g) = \max_i I_i(g)$ and
assume w.l.o.g. $j=n$. Define $f^0,f^1: \Omega_{n-1} \ra [-1,1]$ by
\[
f^0(x_1,\ldots,x_{n-1}) = f(x_1,\ldots,x_{n-1},0) \qquad \mbox{ and }
\qquad f^1(x_1,\ldots,x_{n-1}) = f(x_1,\ldots,x_{n-1},1).
\]
Denote by $a^\ell$ ($\ell=0,1$) the expectation $\E(f^{\ell})$, by $a_j$ ($j \in [n]$) the influence $I_j(f)$,
and by $a^\ell_j$ ($\ell=0,1, j \in [n-1]$) the influence $I_j(f^\ell)$. Define
$g^0,g^1,b^\ell,b_j,b^\ell_j$ in the same way, with $g$ in place of $f$. Since $f^0,f^1$ are fully
symmetric, we have by the induction hypothesis
\begin{align*}
\mathrm{Cov}(f^0,g^0) = \E[f^0 g^0] - a^0b^0 &\geq c \sum_{i=1}^{n-1} \psi_{n-1}(a^0_i) \psi(b^0_i), \qquad \mbox{and} \\
\mathrm{Cov}(f^1,g^1) = \E[f^1 g^1] - a^1b^1 &\geq c \sum_{i=1}^{n-1} \psi_{n-1}(a^1_i) \psi(b^1_i).
\end{align*}
Since $\E[fg] = (\E[f^0 g^0] + \E[f^1 g^1])/2$, we have
\[
\E[fg] - (a^0b^0 + a^1b^1)/2 \geq \frac{c}{2}
\sum_{i=1}^{n-1} \left( \psi_{n-1}(a^0_i) \psi(b^0_i) +
\psi_{n-1}(a^1_i)\psi(b^1_i) \right).
\]
As $\E[f] = (a^0+a^1)/2$ and $\E[g]=(b^0+b^1)/2$, we obtain
\[
\mathrm{Cov}(f,g)=\E[fg]-\E[f]\E[g] \geq \frac{c}{2} \sum_{i=1}^{n-1}
\left( \psi_{n-1}(a^0_i) \psi(b^0_i) + \psi_{n-1}(a^1_i)\psi(b^1_i) \right) +
\frac14(a^1-a^0)(b^1-b^0).
\]
Note that $a^1-a^0=I_n(f)=a_n$ and $b^1-b^0=b_n$, and hence we actually have
\[
\mathrm{Cov}(f,g) \geq \frac{c}{2} \sum_{i=1}^{n-1}
\left( \psi_{n-1}(a^0_i) \psi(b^0_i) + \psi_{n-1}(a^1_i)\psi(b^1_i) \right) +
\frac14 a_n b_n.
\]
Thus, it is sufficient to show
\[
\frac{c}{2} \sum_{i=1}^{n-1} \left( \psi_{n-1}(a^0_i) \psi(b^0_i) +
\psi_{n-1}(a^1_i)\psi(b^1_i) \right) + \frac14 a_n b_n \geq c
\sum_{i=1}^{n} \psi_n(a_i) \psi(b_i),
\]
or equivalently,
\begin{equation}\label{Eq7.3}
c \psi_n(a_n) \psi(b_n) + c \sum_{i=1}^{n-1} \left(\psi_n(a_i)
\psi(b_i) - \frac{1}{2} \left( \psi_{n-1}(a^0_i) \psi(b^0_i) +
\psi_{n-1}(a^1_i)\psi(b^1_i) \right) \right) \leq \frac14 a_n b_n.
\end{equation}
In the next steps, we consider the term $\psi_{n-1}(a_i) \psi(b_i)$ instead of
$\psi_{n}(a_i) \psi(b_i)$, and we shall take care of the difference between them
at a later stage. As for $\ell=0,1$ and for any $i \in [n-1]$, we have
$(n-1)(a_i^\ell)^2 = \sum_i I_i(f^\ell)^2 \leq 1$ by Claim~\ref{Claim:Simple}(c),
we can deduce $a_i^\ell \in (0,1/\sqrt{n-1})$. Hence, we can
use the convexity of $\psi_{n-1}$ and of $\psi$ (see Lemmas~\ref{Lemma:T-Simple2}
and~\ref{Lemma:T-Simple3}) to assert
\[
\psi_{n-1}(a_i) \psi(b_i) = \psi_{n-1}((a_i^0+a_i^1)/2) \psi((b_i^0+b_i^1)/2)
\leq
\frac{1}{4}(\psi_{n-1}(a_i^0)+\psi_{n-1}(a_i^1))(\psi(b_i^0)+\psi(b_i^1)),
\]
for each $i \in [n-1]$. Thus,
\begin{equation}\label{Eq7.4}
\psi_{n-1}(a_i) \psi(b_i) - \frac{1}{2} \left( \psi_{n-1}(a^0_i) \psi(b^0_i) +
\psi_{n-1}(a^1_i)\psi(b^1_i) \right) \leq
\frac{1}{4}(\psi_{n-1}(a_i^0)-\psi_{n-1}(a_i^1))(\psi(b_i^1)-\psi(b_i^0)).
\end{equation}
Now, note that $a^1_i-a^0_i= 2 \hat f(\{i,n\})$, and hence, by Lemma~\ref{Lemma:T-Simple2},
\[
|\psi_{n-1}(a_i^0)-\psi_{n-1}(a_i^1)| \leq \frac{2 \cdot 2 \hat
f(\{i,n\})}{\sqrt{\log \frac{e^3}{(n-1)((a_i^0+a_i^1)/2)^2}}} = \frac{2
\cdot 2 \hat f(\{i,n\})}{\sqrt{\log \frac{e^3}{(n-1)a_i^2}}}.
\]
Similarly, $b^1_i-b^0_i= 2 \hat g(\{i,n\})$, and hence, by Lemma~\ref{Lemma:T-Simple3},
\[
|\psi(b_i^0)-\psi(b_i^1)| \leq \frac{2
\cdot 2 \hat g(\{i,n\})}{\sqrt{\log \frac{e^2}{b_i}}}.
\]
Therefore, from~\eqref{Eq7.4} we get
\begin{equation}\label{Eq7.5}
\sum_{i=1}^{n-1}
\left( \psi_{n-1}(a_i) \psi(b_i) - \frac{1}{2} \left( \psi_{n-1}(a^0_i)
\psi(b^0_i) + \psi_{n-1}(a^1_i)\psi(b^1_i) \right) \right) \leq
\sum_{i=1}^{n-1} \frac{4 |\hat f(\{i,n\}) \hat
g(\{i,n\})|}{\sqrt{\log \frac{e^3}{(n-1)a_i^2}} \sqrt{\log
\frac{e^2}{b_i}}}.
\end{equation}
In the denominator of the right hand side, we can replace $b_i$ by $b_n$ due to the choice of $n$. In the
numerator, we replace $\sum_{i=1}^{n-1} |\hat f(\{i,n\}) \hat g(\{i,n\})|$ by
\[
\left(
\sum_{i=1}^{n-1} \hat f(\{i,n\})^2 \right)^{1/2}  \left(
\sum_{i=1}^{n-1} \hat g(\{i,n\})^2 \right)^{1/2}
\]
using Cauchy-Schwarz, and bound the terms related to $f$ and the terms related to $g$ separately.
For $g$, by Corollary~\ref{Cor:T-Classic} we have
\[
\sum_{i=1}^{n-1} \hat g(\{i,n\})^2 \leq c_1 b_n^2 \log(e/b_n^2),
\]
and for $f$, by the full symmetry of $A$ we can use Lemma~\ref{Lemma:Talagrand}
to get
\[
\sum_{i=1}^{n-1} \hat f(\{i,n\})^2 = \frac{2}{n} \W_2(f) \leq \frac{c}{n} \W_1(f)
\log(e/\W_1(f)) = c_2 a_n^2 \log(e/n a_n^2).
\]
Combining the bounds and summing over $i$, we obtain
\[
\left(
\sum_{i=1}^{n-1} \hat f(\{i,n\})^2 \right)^{1/2}  \left(
\sum_{i=1}^{n-1} \hat g(\{i,n\})^2 \right)^{1/2} \leq \sqrt{c_1 c_2} a_n b_n
\sqrt{\log{\frac{e}{n a_n^2}}}\sqrt{\log{\frac{e}{b_n}}}.
\]
Substituting into~\eqref{Eq7.5} yields
\begin{align}\label{Eq7.7}
\begin{split}
\sum_{i=1}^{n-1}
\left( \psi_{n-1}(a_i) \psi(b_i) - \frac{1}{2} \left( \psi_{n-1}(a^0_i)
\psi(b^0_i) + \psi_{n-1}(a^1_i)\psi(b^1_i) \right) \right) &\leq
\frac{4 \sqrt{c_1 c_2} a_n b_n \sqrt{\log{\frac{e}{n a_n^2}}} \sqrt{\log{\frac{e}{b_n}}} }
{\sqrt{\log \frac{e^3}{(n-1)a_i^2}} \sqrt{\log
\frac{e^2}{b_n}}} \\
&\leq 4 \sqrt{c_1 c_2} a_n b_n.
\end{split}
\end{align}
As $\psi_n(a_n) \psi(b_n) \leq a_n b_n$, this almost proves~(\ref{Eq7.3}), and thus the
theorem. In order to complete the proof, we only have to ``replace'' $\psi_{n-1}(a_i)$ which we used
in our argument with $\psi_n(a)$. This is done using the following calculation:
\begin{align}\label{Eq7.8}
\begin{split}
\left| \sum_{i=1}^{n-1} \psi_n(a_i)\psi(b_i) - \psi_{n-1}(a_i)\psi(b_i) \right| &\leq
b_n \sum_{i=1}^{n-1} |\psi_n(a_i) - \psi_{n-1}(a_i)| \\
&\leq n a_n b_n  \left(\left(\log(e^3/n a_n^2) \right)^{-1/2}
- \left(\log(e^3/(n-1) a_n^2) \right)^{-1/2} \right).
\end{split}
\end{align}
Since for any $x,y>1$ we have $x^{-1}-y^{-1} \leq y-x \leq y^2-x^2$, we obtain
\[
\left(\left(\log(e^3/n a_n^2) \right)^{-1/2}
- \left(\log(e^3/(n-1) a_n^2) \right)^{-1/2} \right) \leq \log(e^3/(n-1) a_n^2) - \log(e^3/n a_n^2) =
\log(\frac{n}{n-1}).
\]
Substituting into~(\ref{Eq7.8}) yields
\begin{equation}\label{Eq7.9}
\left| \sum_{i=1}^{n-1} \psi_n(a_i)\psi(b_i) - \psi_{n-1}(a_i)\psi(b_i) \right| \leq n a_n b_n \log(\frac{n}{n-1}) \leq 2 a_n b_n.
\end{equation}
Combining Equations~(\ref{Eq7.7}) and~(\ref{Eq7.9}), we obtain
\begin{equation}
\sum_{i=1}^{n-1}
\left( \psi_n(a_i) \psi(b_i) - \frac{1}{2} \left( \psi_{n-1}(a^0_i)
\psi(b^0_i) + \psi_{n-1}(a^1_i)\psi(b^1_i) \right) \right) \leq (4 \sqrt{c_1 c_2} +2) a_n b_n,
\end{equation}
which implies that Equation~(\ref{Eq7.3}) holds with $c=1/(4\sqrt{c_1 c_2}+3)$, completing the
proof.

\begin{remark}
We note that without the full symmetry assumption on $f$, we can use the same argument (using $\psi$
for both functions) to obtain an alternative proof of Theorem~\ref{Thm:Alt_Bound}. The original
proof presented in~\cite{KMS14} is rather different, using a reduction from the Gaussian case and
the so-called \emph{reverse isoperimetric inequality} of Borell~\cite{Borell82}.
\end{remark}

\nin As demonstrated by Example~2.2 above, the full symmetry assumption on $f$ cannot be replaced
by a regularity assumption. It will be interesting to find less restrictive sufficient conditions
for Theorem~\ref{Thm:Imp_Bound_Asymmetric}.

\section{Tightness of Theorem~\ref{Thm:Talagrand-Correlation}}
\label{sec:examples}

In this section we present several new tightness examples of Theorem~\ref{Thm:Talagrand-Correlation}.
We present a few simple sufficient conditions and one necessary condition for tightness
of~\eqref{Eq:Talagrand}, and then we give several concrete examples. Throughout the
section, we use the notation $\E'[h]=\min(\E[h],1-\E[h])$ for any
$h:\Omega_n \ra [0,1]$, $\E''[h]=\min(1-\E[h],\E[h]+1)$ for any $h:\Omega_n \ra [-1,1]$,
and $\mu'(\C) = \min(\mu(\C),1-\mu(\C))$ for any family $\C$.

\subsection{Conditions for tightness of Theorem~\ref{Thm:Talagrand-Correlation}}

We start with a simple necessary condition, which states that~\eqref{Eq:Talagrand} can
be tight only if the correlation of $f,g$ is rather small.
\begin{proposition}
Theorem~\ref{Thm:Talagrand-Correlation} may be tight for $f,g: \Omega \ra [0,1]$ only if
$\Cov(f,g) = O(\E'[f]\E'[g])$.
\end{proposition}

\begin{proof}
Clearly, it is sufficient to prove that the right hand side of~\eqref{Eq:Talagrand} is
at most $O(\E'[f]\E'[g])$. By Cauchy-Schwarz and the inequality between the arithmetic
and geometric means, we have
\begin{align*}
\sum_{i=1}^n \frac{I_i(f) I_i(g)}{\log (e/\sum_{i=1}^n I_i(f) I_i(g))}
&= \frac{\W_1(f,g)}{\log(e/\W_1(f,g))} \leq \frac{\sqrt{\W_1(f)\W_1(g)}}{0.5\log(e^2/\W_1(f)\W_1(g))} \\
&\leq \frac{\sqrt{\W_1(f)\W_1(g)}}{\sqrt{\log(e/\W_1(f))\log(e/W_1(g))}} =
\sqrt{\varphi(\W_1(f)) \varphi(\W_1(g))},
\end{align*}
where $\varphi(x) = x/\log(e/x)$  as above. By Lemma~\ref{Lemma:T-Classic}, we have
$\W_1(f) \leq c \E'[f]^2 \log(e/\E'[f])$ and similarly for $g$. As $\varphi$ is
increasing (Lemma~\ref{Lemma:T-Simple}), we obtain
\[
\sqrt{\varphi(\W_1(f)) \varphi(\W_1(g))} \leq \sqrt{\varphi(c \E'[f]^2 \log(e/\E'[f]))}
\sqrt{\varphi(c \E'[g]^2 \log(e/\E'[g]))} \leq c' \E'[f]\E'[g],
\]
completing the proof.
\end{proof}

\mn Our first sufficient condition is also related to Lemma~\ref{Lemma:T-Classic}.
\begin{notation}
An increasing family $\A \subset \Omega_n$ is called \emph{first-level optimal}
if it is a tightness example (up to a constant) for Lemma~\ref{Lemma:T-Classic}, that is,
if $\W_1(\1_\A) \geq c_0 \E'[\1_\A]^2 \log(e/\E'[\1_\A])$ for a universal constant $c_0$.
First-level optimality of a function $f:\Omega_n \ra [-1,1]$ is defined similarly,
with $\E''[f]$ in place of $\E'[\1_\A]$.
\end{notation}
\nin As usual, the formally correct definition is to consider a family of families $\{\A_m \subset \Omega_m\}$,
with an asymptotic property $\W_1(\1_{\A_m}) = \Omega(\E'[\1_{\A_m}] \log(e/\E'[\1_{\A_m}]))$. For
sake of simplicity, we treat a single family $\A=\A_n$ and assume that $n$ is sufficiently large.
\begin{proposition}\label{Prop:Suf1}
If $\A$ is first-level optimal and $\B$ is the dual of $\A$ then~\eqref{Eq:Talagrand} is tight
for $(\A,\B)$ (up to the constant $c$).
\end{proposition}

\begin{proof}
We have to show that $\Cov(\A,\B) \leq c \varphi(\W_1(\1_\A,\1_\B))$. Since $I_i(\A)=I_i(\B)$
for all $i$, we have
\begin{equation}\label{Eq:Optimal1}
\varphi(\W_1(\1_\A,\1_\B)) = \varphi(\W_1(\1_\A)) \geq \varphi(c_0 \E'[\1_\A] \log(e/\E'[\1_\A]))
\geq c' \mu'(\A),
\end{equation}
the penultimate inequality using the first-level optimality of $\A$ and monotonicity of $\varphi$.
On the other hand,
\begin{equation}\label{Eq:Optimal2}
\Cov(\A,\B) = \mu(\A \cap \B)-\mu(\A)\mu(\B) \leq \min(\mu(\A),\mu(B)) = \mu'(\A),
\end{equation}
the last equality using $\mu(\B)=1-\mu(\A)$. Comparing~\eqref{Eq:Optimal1} and~\eqref{Eq:Optimal2}
completes the proof.
\end{proof}

\mn The second sufficient condition is a simple composition lemma.
\begin{notation}
Let $f:\Omega_n \ra \R$ and let $g_1,g_2,\ldots,g_n:\Omega_m \ra \{0,1\}$. The composition
$f \circ (g_1,\ldots,g_n): \Omega_{mn} \ra \R$ is defined by
\begin{align*}
f \circ (g_1,\ldots,g_n) &\left(x^1_1,\ldots,x^1_m,x^2_1,\ldots,x^2_m,\ldots,x^n_1,\ldots,x^n_m \right)= \\
&= f \left(g_1(x^1_1,\ldots,x^1_m),g_2(x^2_1,\ldots,x^2_m),\ldots,g_n(x^n_1,\ldots,x^n_m) \right).
\end{align*}
\end{notation}

\begin{proposition}\label{Prop:Suf2}
Let $(f_1,f_2)$, with $f_1,f_2: \Omega_n \ra \{0,1\}$, be a tightness example for~\eqref{Eq:Talagrand},
and let $g_1,\ldots,g_n: \Omega_m \ra \{0,1\}$ be increasing functions such that
\begin{itemize}
\item $\E[g_i]=1/2$ for all $i$, and
\item $\W_1(g_i) \geq c_0$ for all $i$, where $c_0$ is a universal constant.
\end{itemize}
Then $(f_1 \circ (g_1,\ldots,g_n), f_2 \circ (g_1,\ldots,g_n))$ is also a tightness example
for~\eqref{Eq:Talagrand} (though, with a different constant).
\end{proposition}

\begin{proof}
For $\ell=0,1$, denote $\tilde{f_\ell} = (f_\ell \circ (g_1,\ldots,g_n))$. It is clear that
$\E[\tilde{f_\ell}] = \E[f_\ell]$, and $\E[\tilde{f_1} \tilde{f_2}] = \E[f_1 f_2]$ (here we
use the fact that $g_1,\ldots,g_n$ are the same for $f_1,f_2$). Hence,
$\Cov(\tilde{f_1},\tilde{f_2})=\Cov(f_1,f_2)$. On the other hand, denoting by $I_{i,j}(f_\ell)$
the influence on the variable $x^i_j$ on $f_\ell$, we have $I_{i,j}(f_\ell) = I_i(f_\ell) I_j(g_i)$.
Thus,
\[
\W_1(\tilde{f_1},\tilde{f_2})=\sum_i I_i(f_1)I_i(f_2) \sum_j I_j(g_i)^2 \geq c_0 \sum_i I_i(f_1)I_i(f_2) =
c_0 \W_1(f_1,f_2),
\]
where the inequality uses the assumption on $\{g_i\}$. Therefore,
$\varphi(\W_1(\tilde{f_1},\tilde{f_2})) \geq c \varphi(\W_1(f_1,f_2))$, completing the proof.
\end{proof}

\subsection{A few properties of linear threshold functions}

Before we present the specific examples, we cite a few definitions and results on
\emph{linear threshold functions} that will be a central ingredient of the examples.
\begin{definition}
A \emph{linear threshold function} is $f:\Omega_n \ra \{-1,1\}$ of the form
$f(x_1,\ldots,x_n) = \mathrm{sign}(\sum a_i x_i - \theta)$, where $a_i,\theta \in \R$ and
$\mathrm{sign}(x)=1$ if $x \geq 0$ and $\mathrm{sign}(x)=-1$ otherwise.
\end{definition}
\nin Linear threshold functions (LTFs) are a central object of study in computer science (see,
e.g.,~\cite{O'Donnell14}). It is clear that an LTF is increasing iff $a_i \geq 0$ for all $i$,
and balanced (i.e., satisfies $\E[f]=0$) iff $\theta=0$. The next definition captures
the notion of \emph{low-influence} functions.
\begin{definition}
A function $f:\Omega_n \ra \R$ is called $\tau$-regular if $I_i(f) \leq \tau ||f||_2$ for all $i$.
\end{definition}
\nin Intuitively, having low influences allows to approximate the function by a Gaussian via the
Central Limit Theorem and to use Gaussian tools to handle it (see, e.g.,~\cite{MORS10}).
\begin{notation}
For $x \in (0,1)$, let $u(x) = 2[\phi(\Phi^{-1}(x))]^2$, where $\phi$ is the density function and
$\Phi$ the cumulative distribution function of a Gaussian $N(0,1)$ random variable.
\end{notation}
\nin It is easy to see that if $x=1-\eta$, then $u(x)=\Theta(\eta^2 \log(1/\eta))$
(see~\cite[Proposition 24]{MORS10}). We use the following theorem of Matulef et al.~\cite{MORS10}.
\begin{theorem}(~\cite[Theorem 48]{MORS10})\label{Thm:MORS10}
Let $f_1(x_1,\ldots,x_n) = \mathrm{sign}(\sum a_i x_i - \theta_1)$ with $\sum a_i^2 =1$ be
a $\tau$-regular LTF. Then
\[
\left|\W_1(f_1) - u(\E[f_1]) \right| \leq \tau^{1/6}.
\]
Furthermore, if $f_2(x_1,\ldots,x_n) = \mathrm{sign}(\sum a_i x_i - \theta_2)$ is another LTF
with the same weights $a_i$ then
\[
\left|\W_1(f_1,f_2)^2 - u(\E[f_1]) u(\E[f_2]) \right| \leq \tau^{1/6}.
\]
\end{theorem}
\nin An immediate corollary of Theorem~\ref{Thm:MORS10} is that if $f$ is $\tau$-regular, where
$\tau \leq c(\E''[f]^2 \log(1/\E''[f]))^6$ for a sufficiently small $c$, then $f$ is \emph{first-level
optimal}.

\nin For balanced LTFs, we can deduce the same conclusion without the $\tau$-regularity assumption,
using the following theorem of Peres~\cite{Peres04} (which shows that LTFs are \emph{asymptotically noise
stable}, see~\cite{BKS}):
\begin{theorem}\cite{Peres04}\label{Thm:Peres}
Let $f:\Omega_n \ra \{-1,1\}$ be a balanced LTF. Then
\[
NS_\epsilon(f):= \frac12 - \frac12 \sum_{S \subset [n]} (1-2\epsilon)^{|S|} \hat f(S)^2 \leq O(\sqrt{\epsilon}).
\]
\end{theorem}
\nin Theorem~\ref{Thm:Peres} immediately implies that balanced LTFs are \emph{first-level optimal} (using,
e.g.,~\cite[Theorem~4]{KK13}).

\subsection{Specific examples}

Recall that there are two previously known examples: $\A$ being a small Hamming ball and $\B$
being its dual (presented by Talagrand)~\cite{Talagrand96}, and $\A$ being the tribes function and $\B$ being its
dual~\cite{Keller09}.

\mn \textbf{Example 5.1.} Our first example is an extension of Talagrand's example.
\begin{proposition}\label{Prop:Ex1}
Let $f$ be an increasing $\tau$-regular LTF, with $\tau \leq c(\E''[f]^2 \log(1/\E''[f]))^6$ for a
sufficiently small $c$. Let $\A \subset \Omega_n$ be a family such that $f=2 \cdot \1_\A-1$, and let $\B$
be the dual of $\A$. Then~\eqref{Eq:Talagrand} is tight for $(\A,\B)$.
\end{proposition}

\begin{proof}
As mentioned above, Theorem~\ref{Thm:MORS10} implies that $\A$ is first-level optimal.
The assertion now follows from Proposition~\ref{Prop:Suf1}.
\end{proof}
\nin Talagrand's example is a special case, with
$f=\mathrm{sign}(\sum_i \frac{1}{\sqrt{n}} x_i - \theta)$, for any $\theta$ such that
$\E''[f]$ is not too small. It is plausible that Proposition~\ref{Prop:Ex1} actually holds
for \emph{any} LTF (i.e., without the $\tau$-regularity assumption), which would yield a wider
class of tightness examples.

\mn \textbf{Example 5.2.} The second example is a generalization of a \emph{layered majority
function} with a constant number of layers. For simplicity of notation, we replace our domain
$\Omega_n$ by $\Omega'_n = \{-1,1\}^n$.
\begin{definition}
A $1$-layer weighted majority function is an increasing balanced LTF (on any number of coordinates,
including a single coordinate). A $k$-layer weighted majority
function is defined inductively as $f \otimes (g_1,g_2,\ldots,g_{n})$, where
$f:\Omega'_n \ra \{-1,1\}$ is an increasing balanced LTF and
$g_1,\ldots,g_{n}$ are $k-1$-layer weighted majority functions.
\end{definition}

\begin{proposition}\label{Prop:Ex2}
Let $k \in \mathbb{N}$ be constant, let $(f_1,f_2)$ be a pair of functions on $\Omega'_n$ for
which~\eqref{Eq:Talagrand} is tight, and let $g_1,\ldots,g_n$ be layered majority functions
with at most $k$ layers. Then~\eqref{Eq:Talagrand} is tight for the functions
$(f_1 \circ (g_1,\ldots,g_n), f_2 \circ (g_1,\ldots,g_n))$.
\end{proposition}

\begin{proof}
As mentioned above, Theorem~\ref{Thm:Peres} implies that any balanced increasing LTF $g$ satisfies
$\W_1(g) \geq c$. By induction on $k$, the same holds for any $k$-layer weighted majority (with
a constant that depends on $k$). Since $k$ is assumed to be constant, the assertion follows
from Proposition~\ref{Prop:Suf2}.
\end{proof}

\mn \textbf{Example 5.3.} The two example classes presented above consist of a family and its
dual, as the previously known examples. A conceptually different type of examples is those
presented in Section~\ref{sec:sub:counterexamples}. For sake of completeness, we restate them here.

\begin{proposition}
For a small constant $a$, let $\A = m_a(x_1,\ldots,x_n)$, and let $\B = m_{1-a}(x_1,\ldots,x_n)$
be the dual of $\A$. Then the following pairs are tightness examples for~\eqref{Eq:Talagrand}.

\mn
{\rm (a)} $\A', \B' \subset \Omega_{n+1}$, defined by
\[
\A' = \{(x_1,\ldots,x_n,y): ((x_1,\ldots,x_n) \in \A) \vee (y=1)\}, \qquad \mbox{and} \qquad \B' = m_{1-a}(x_1,\ldots,x_n,y).
\]

\mn
{\rm (b)} $\A',\B' \subset \Omega_{n+\ell}$, defined by
$\A' = \{(x_1,\ldots,x_n,y_1,\ldots,y_\ell): ((x_1,\ldots,x_n) \in \A) \vee (y_1=1)\}$ and
$\B' = \{(x_1,\ldots,x_n,y_1,\ldots,y_\ell): ((x_1,\ldots,x_n) \in \B) \wedge (y_1,\ldots,y_\ell) \in \C\}$,
where $\C = m_{1/2}(y_1,\ldots,y_\ell)$ and $\ell=\ell(n)$ is chosen such that
$I_i(\C) = I_j(\B)$ for all $i,j$.
\end{proposition}

\begin{proof}
The tightness of~\eqref{Eq:Talagrand} for both pairs of examples follows immediately from the
computations presented in Section~\ref{sec:sub:counterexamples}.
\end{proof}

\mn Our concluding example is \emph{not} a tightness example of Theorem~\ref{Thm:Talagrand-Correlation}, but rather
provides a case study for comparing all lower bounds considered in the paper.

\mn \textbf{Example 5.4.} Let $f_1 = \mathrm{sign}(\sum a_i x_i - \theta)$ be a $\tau$-regular LTF,
with $E[f]=1-a$ for a small constant $a$ and $\tau \leq c(\E''[f]^2 \log(1/\E''[f]))^{12}$ for a
sufficiently small $c$. Let $f_2 = \mathrm{sign}(\sum a_i x_i)$.

\mn Since $f_1 f_2(x)=f_2(x)$ for all $x$, we have $\Cov[f_1 f_2] = \E[f_2] - \E[f_1]\E[f_2] = a/2$.
On the other hand, by Theorem~\ref{Thm:MORS10}, we have
\[
\W_1(f_1,f_2) = \Theta(a \sqrt{\log(1/a)}), \qquad \W_1(f_1) = \Theta(a^2 \log(1/a)), \qquad \mbox{and}
\qquad \W_1(f_2) = \Theta(1).
\]
Hence, for the pair of functions $(f_1,f_2)$ the bound~\eqref{Eq:Imp_Bound1} is tight, while the bounds
~\eqref{Eq:Talagrand} and~\eqref{Eq:Wrong_Bound1} are off by a factor of $\Theta(\sqrt{\log(1/a)})$.
In the specific case of $f_1$ corresponding to a Hamming ball, i.e.,
$f_1 = \mathrm{sign}(\sum \frac{1}{\sqrt{n}} x_i -\theta)$, we can compute also the
bounds~\eqref{Eq:Alt_Bound} and~\eqref{Eq:Imp_Bound2} and find that~\eqref{Eq:Alt_Bound} is off
by a factor of $\log n/\sqrt{\log(1/a)}$, while~\eqref{Eq:Imp_Bound2} is off by a factor of
$\sqrt{\log(n)/\log(1/a)}$.

This example demonstrates the advantage of Theorem~\ref{Thm:Imp_Bound_Symmetric} over all other
bounds we consider. Note however that while Theorem~\ref{Thm:Imp_Bound_Symmetric} holds for
$f_1 = \mathrm{sign}(\sum \frac{1}{\sqrt{n}} x_i -\theta)$, we do not know whether it can be
generalized to any low-influence LTF. We do know that it does not hold for LTFs in general
(Example~3.1 being a  counterexample), but it seems plausible that it should hold under an
appropriate $\tau$-regularity assumption.

\section{Open Problems}
\label{sec:open}

We conclude this paper with a few open problems.

\mn \textbf{Problem 6.1.} A much stronger, and more ``nice-looking'', correlation lower
bound is
\begin{equation}\label{Eq:Wrong2}
%\Cov(\A,\B) \geq \sum_i I_i(\A) I_i(\B).
\Cov(f,g) \geq \sum_i I_i(f) I_i(g).
\end{equation}
It was shown in~\cite{Keller09} that~\eqref{Eq:Wrong2} holds ``on average'', i.e., when correlation is
averaged over all pairs of elements in a family $\T$. While it clearly does not hold in general (all
examples of Section~\ref{sec:examples} being counterexamples), it will be interesting to find
additional conditions under which~\eqref{Eq:Wrong2} holds, both for Boolean functions and for general functions.
One condition that may be relevant is the submodularity condition which is of great interest in
combinatorics and optimization. As shown in~\cite{FKKK15}, for families of sets Equation~\eqref{Eq:Wrong2},
as well as several weaker correlation inequalities,
%implies a special case of
are related to a conjecture of Chv\'{a}tal in extremal set theory.

\mn \textbf{Problem 6.2.} It will be interesting to understand in which cases Lemma~\ref{Lemma:T-Classic}
is tight. That is, what are the families $\A$ that satisfy
\begin{equation}
\label{Eq:Level-1-tight}
\sum_i I^2_i(\A) \geq c \mu(\A)^2 \log(1/\mu(\A)),
\end{equation}
for a universal constant $c$. This question seems to be of independent interest, due to the abundance
of applications of Lemma~\ref{Lemma:T-Classic}, and also can provide more tightness examples for
Theorem~\ref{Thm:Talagrand-Correlation} (using Proposition~\ref{Prop:Suf1}). In~\cite{MORS10},
it is shown that if $\W_1(f)$ is very close to the maximum possible, then $f$ must be a linear threshold
function. However, when we ask for tightness only up to a constant factor, the question looks harder.
One specific case that may be easy to handle is to show that~\eqref{Eq:Level-1-tight} holds
for \emph{any} LTF (and not only for low-influence LTFs as shown in Theorem~\ref{Thm:MORS10}).

\mn \textbf{Problem 6.3.}
It will be interesting to find additional conditions under which
Theorem~\ref{Thm:Imp_Bound_Symmetric} holds, i.e.,
\begin{equation}\label{Eq:Imp_Bound_Open}
\Cov(\A,\B) \geq c \frac{\W_1(\1_{\A},\1_{\B})}{\sqrt {\log(e/\W_1(\1_\A))}{\sqrt {\log(e/\W_1(\1_\B))}}},
\end{equation}
for a universal constant $c$. In particular, it seems plausible that~\eqref{Eq:Imp_Bound_Open} holds
for any pair of low-influence LTFs. If true, this will provide an additional tightness example of
Theorem~\ref{Thm:Imp_Bound_Symmetric}, in a case where all other bounds considered in this paper are not tight
(see Example~4.4.).

%\mn \textbf{Problem 5.4.}
%\begin {conjecture}
%\begin {equation}
%\Cov(\A,\B) \geq c {\W_1(\1_{\A},\1_{\B})} \sqrt {\frac {\int_0^1 \|T_\rho(\A)\|_2}{\|\A\|_2}} \sqrt {\frac {\int_0^1 \|T_\rho(\B)\|_2}{\|\B\|_2}}.
%\end {equation}
%\end {conjecture}

\section{Acknowledgements}

We are grateful to Ryan o'Donnell for suggesting to use Theorem~\ref{Thm:MORS10} to provide
tightness examples for Lemma~\ref{Lemma:T-Classic}.

\end{document}